\documentclass[12pt]{amsart}
\usepackage[all]{xy}
\usepackage{amssymb,amsmath,amsthm}
\usepackage{cases}
\usepackage{setspace,shuffle}
\usepackage{hyperref}
\usepackage{enumerate}

\numberwithin{equation}{section}
\allowdisplaybreaks[3]
\makeatother
\newtheorem{thm}{Theorem}[section]

\newtheorem{lem}[thm]{Lemma}

\theoremstyle{definition}

\newtheorem{rem}[thm]{Remark}

\numberwithin{equation}{section}

\theoremstyle{definition}
\newtheorem*{ack}{Acknowledgements}

\newcommand{\C}{\mathbb{C}}
\newcommand{\N}{\mathbb{N}}
\newcommand{\Z}{\mathbb{Z}}

\newcommand{\lab}{\label}

\newcommand{\eps}{\varepsilon}
\newcommand{\ph}{\varphi}
\newcommand{\al}{\alpha}
\newcommand{\om}{\omega}

\newcommand{\la}{\lambda}
\newcommand{\si}{\sigma}
\newcommand{\Gm}{\Gamma}

\newcommand{\re}{\textup{Re}}

\newcommand{\Ll}{\mathcal{L}}
\newcommand{\G}{\mathcal{G}}

\begin{document}

\title[The Mordell-Tornheim multiple Dirichlet series]{A relation between the Mordell-Tornheim multiple Dirichlet series and the confluent hypergeometric function}

\author[Y. Toma]{Yuichiro Toma}
\address{Graduate School of Mathematics, Nagoya University, Chikusa-ku, Nagoya 464-8602, Japan.}

\email{m20034y@math.nagoya-u.ac.jp}  


\begin{abstract} We investigate the expressions of the Mordell-Tornheim multiple Dirichlet series in terms of the confluent hypergeometric function. We prove them by applying the Mellin-Barnes integral formula. Moreover, in the double case, our results include the functional equations for two kinds of double $L$-functions shown by Choie and Mastumoto, and by Komori, Matsumoto and Tsumura, respectively. 
\end{abstract}

\maketitle
\section{Introduction}
Let $s = \si + it$ be a complex variable. In this paper we introduce the following Mordell-Tornheim type of $r$-fold Dirichlet series
\begin{equation}
\lab{formula:MTML}
\Ll_{MT,r} (s_1, \dots,s_r,s_{r+1}; a_1, \dots, a_r) = \sum_{m_1,\dots, m_r\geq 1} \frac{a_1(m_1) \dots a_r(m_r)}{m_1^{s_1}\dots m_r^{s_r} (m_1+\dots+m_r)^{s_{r+1}}},
\end{equation}
where the sequence of complex numbers $\{ a_k (n) \}_{n \geq 1}$ satisfies
\begin{enumerate}
\item\lab{axiom1} $a_k(n) \ll n^{\al_k+\eps}$ ($1 \leq k \leq r$) with a certain constant $\al_k \geq 0$, where $\eps$ is an arbitrary small positive number,
\item\lab{axiom2} the Dirichlet series 
\begin{equation*}
\Ll_k (s) = \Ll (s,a_k) = \sum_{n=1}^\infty \frac{a_k (n)}{n^s} \quad (1 \leq k \leq r)
\end{equation*}
can be continued meromorphically to the whole $\C$-plane, and holomorphic except for finitely many possible poles, 
\item\lab{axiom3} for $2 \leq k \leq r$, the order estimate $\Ll_k (\si + it) = O(\lvert t \rvert^A)$ holds as $\lvert t \rvert \to \infty$ in any fixed strip $\si_{1_k} \leq  \si \leq \si_{2_k}$, where $A$ is a non-negative constant which depends on $\si_{1_k}$ and $\si_{2_k}$.
\end{enumerate}
We note that from (\ref{axiom1}), the Dirichlet series $\Ll_k (s) = \sum_{n=1}^\infty a_k (n) n^{-s}$ is absolutely convergent for $\si > \al_k +1$. If $a_k =1$ for all $k$, then $\Ll_1 ,\dots,\Ll_r$ are the Riemann zeta-function, which is defined by $\zeta(s) =  \sum_{n-1}^\infty n^{-s}$ for $\si>1$. In this case, the series (\ref{formula:MTML}) is nothing but the Mordell-Tornheim multiple zeta-function which is defined by 
\begin{equation}
\lab{formula:MT-MZF}
\zeta_{MT,r} (s_1, \dots, s_r, s_{r+1}) =\sum_{m_1, \dots,m_r=1}^\infty \frac{1}{m_1^{s_1} \dots m_r^{s_r}(m_1+ \dots +m_r)^{s_{r+1}}}.
\end{equation}
This multiple series (\ref{formula:MT-MZF}) converges absolutely when 
\begin{equation*}
\sum_{\ell=1}^j \si_{k_\ell} + \si_{r+1} >j,
\end{equation*}
with $1 \leq k_1<k_2<\dots<k_j \leq r$ for any $j=1,2,\dots,r$ (\cite[Lemma 2.1]{OO}). Further, Matsumoto \cite{Ma03} proved that (\ref{formula:MT-MZF}) can be continued meromorphically to the whole $\C^{r+1}$ by applying the Mellin-Barnes integral formula (see (\ref{formula:MBI})) and he determined the possible singularities of (\ref{formula:MT-MZF}). In addition, if all $a_k$ are the Dirichlet characters $\chi_k$ modulo $q>1$, then (\ref{formula:MTML}) coincides with 
\begin{equation}
\lab{formula:L_MT}
L_{MT,r} (s_1, \dots, s_r, s_{r+1}; \chi_1, \dots, \chi_r) = \sum_{m_1, \dots,m_r=1}^\infty \frac{\chi_1(m_1) \dots\chi_r(m_r)}{m_1^{s_1} \dots m_r^{s_r}(m_1+ \dots +m_r)^{s_{r+1}}}.
\end{equation}
This series was introduced by Wu in his unpublished master thesis \cite{Wu03}. In \cite{Wu03} (see also \cite[Theorem 3]{Ma06}), he proved that this series can be continued meromorphically to $\C^{r+1}$. If none of the characters $\chi_1, \dots,\chi_r$ are principal, then $L_{MT,r}$ is entire. If there are $k$ principal characters $\chi_{j_1},\dots \chi_{j_k}$ among them, then possible singularities are located only on the subsets of $\C^{r+1}$ defined by one of the following
equations:
\begin{equation*}
\sum_{a=1}^h s_{j_{i(a)}} +s_{r+1} = h -\ell \left( 1-\left[\frac{h}{r}\right] \right),
\end{equation*}
where $1 \leq h \leq k, 1 \leq i(1)<\dots <i(h)\leq k$ and $\ell \in \N \cup \{0\}$.

We have to note that a different type of multiple Dirichlet series has been studied by Tanigawa and Matsumoto \cite{MaTa03}. In \cite{MaTa03}, they considered the following multiple Dirichlet series 
\begin{equation*}
\Phi_r (s_1,\dots,s_r;\ph_1, \dots,\ph_r) = \sum_{m_1, \dots,m_r=1}^\infty \frac{a_1(m_1) a_2(m_2) \dots a_r(m_r)}{m_1^{s_1} (m_1+m_2)^{s_2} \dots (m_1+\dots+m_r)^{s_r}},
\end{equation*}
where $\ph_k (s) = \sum_{n\geq 1} a_k(n) n^{-s}$ ($1 \leq k \leq r$) satisfies some axioms, and showed that $\Phi_r (s_1,\dots,s_r;\ph_1, \dots,\ph_r)$ can be continued meromorphically to the whole $\C^r$ space, and determined its possible singularities.

The major purpose of this paper is to investigate the expressions of (\ref{formula:MTML}) in terms of the confluent hypergeometric function. For the case of the Euler-Zagier double zeta-function, which is defined by
\begin{equation}
\lab{formula:EZ-DZF}
\zeta_{EZ,2}(s_1,s_2) = \sum_{m=1}^\infty \sum_{n=1}^\infty \frac{1}{m^{s_1}(m+n)^{s_2}},
\end{equation}
it is known (for example \cite{Ka93}, \cite{Ma98}) that (\ref{formula:EZ-DZF}) has an expression in terms of the confluent hypergeometric function
\begin{equation}
\lab{confluent hypergeometric function} 
\Psi(a,c;x) = \frac{1}{\Gm(a)} \int_0^{\infty e^{i \phi}} e^{-xy} y^{a-1}(1+y)^{c-a-1} dy,
\end{equation}
where $\re(a)>0, -\pi <\phi <\pi, \lvert\phi+ \arg x\rvert<\pi/2$ (\cite[6.5 (2)]{ErMaObTr}). By using this property, Matsumoto \cite{Ma04} showed the functional equations for (\ref{formula:EZ-DZF}). 

Recently, Okamoto and Onozuka \cite{OO} studied analytic properties of (\ref{formula:MT-MZF}). They applied Matsumoto's method in \cite{Ma04} to (\ref{formula:MT-MZF}) and obtained an expression of (\ref{formula:MT-MZF}) in terms of (\ref{confluent hypergeometric function}) for general depth $r$.

Define a modified function
\begin{align*}
g(s_1,\dots,s_r,s_{r+1}) &= \zeta_{MT,r}(s_1,\dots,s_r,s_{r+1}) \\
& -\frac{\Gm(1-s_r) \Gm(s_r+s_{r+1}-1)}{\Gm(s_{r+1})} \zeta_{MT,r-1} (s_1,\dots,s_{r-1}, s_r+s_{r+1}-1).
\end{align*}
Also we put $\sigma_\alpha (k) =\sum_{d \mid k} d^\alpha$. Then Okamoto and Onozuka \cite{OO} gave the following formula:

\begin{thm}[{\cite[Theorem 1.2]{OO}}]
\lab{thm:FEforMTZF}
We have
\begin{align*}
& \frac{g(-s_1,\dots,-s_{r-1},1-s_{r+1},1-s_r)}{i^{s_r+s_{r+1}-1} \Gm(s_{r+1})} \\
&\quad +e^{\frac{\pi i}{2} (s_r+s_{r+1}-1)} F_r^+ (s_1,\dots,s_r,s_{r+1}) + e^{-\frac{\pi i}{2} (s_r+s_{r+1}-1)} F_r^- (s_1,\dots,s_r,s_{r+1}) \\
&= \frac{g(s_1,\dots,s_{r-1},s_r,s_{r+1})}{(2\pi)^{s_r+s_{r+1}-1} \Gm(1-s_r)} \\
&\quad+e^{-\frac{\pi i}{2} (s_r+s_{r+1}-1)} \sum_{\ell_1,\dots,\ell_{r-1} =1}^\infty \si_{MT,r-1} (s_1,\dots,s_{r-1},s_r+s_{r+1}-1;\ell_1,\dots,\ell_{r-1}) \\
& \times \left\{ \Psi (s_{r+1},s_r+s_{r+1}; 2\pi i (\ell_1+\dots+\ell_{r-1})) \right.\\
&\qquad \left. + \Psi (s_{r+1},s_r+s_{r+1}; -2\pi i (\ell_1+\dots+\ell_{r-1})) \right\},
\end{align*}
where
\begin{align}
\lab{formula:F^pm}
F^\pm (s_1,\dots,s_r,s_{r+1}) &= \sum_{\ell_1,\dots,\ell_{r-1} =1}^\infty \frac{\si_{s_1+\dots+s_r+s_{r+1}-1} (\gcd(\ell_1,\dots,\ell_{r-1}))}{\ell_1^{s_1} \dots \ell_{r-1}^{s_{r-1}}} \\
& \times \Psi (s_{r+1},s_r+s_{r+1}; \pm 2\pi i (\ell_1+\dots+\ell_{r-1})) \nonumber
\end{align}
and
\begin{equation}
\si_{MT,r} (s_1,\dots,s_{r-1},s_{r+1};\ell_1,\dots,\ell_r) = \sum_{\substack{d_1 \mid \ell_1,\dots,d_r \mid \ell_r \\ d_j \geq \frac{\ell_j}{\gcd(\ell_1,\dots,\ell_r)}}} d_1^{s_1} \dots d_r^{s_r} (d_1+\dots+d_r)^{s_{r+1}}.
\end{equation}
\end{thm}

\section{Statement of main results}\label{sec2}
Before stating the main results, we introduce the following function. Let 
\begin{align*}
& F_x^\pm (s_1,\dots,s_{r+1} ; a_1, \dots, a_r) \\ 
&= \sum_{\ell_1,\dots,\ell_{r-1}=1}^\infty \sum_{n \mid \gcd(\ell_1,\dots,\ell_{r-1})} \frac{n^{s_1+\dots+s_r+s_{r+1}-1} a_1(\frac{\ell_1}{n}) \dots a_{r-1}(\frac{\ell_{r-1}}{n}) a_r(n)}{\ell_1^{s_1} \dots \ell_{r-1}^{s_{r-1}}} \nonumber \\ 
&\quad \times \Psi(s_{r+1}, s_r+s_{r+1}; \pm 2\pi i(\ell_1+\dots+\ell_{r-1})/x).
\end{align*} 
Then we obtain the following formulas, which are the main results in the present paper.
\begin{thm}
\lab{thm:FEforMTLF1}
We assume that $\Ll_r(s)=\zeta (s)$. Let 
\begin{align*}
&\G_r (s_1,\dots,s_r,s_{r+1} ; a_1, \dots, a_{r-1},1) \\
&= \Ll_{MT,r} (s_1, \dots,s_r,s_{r+1}; a_1, \dots, a_{r-1},1) \\
& -\frac{\Gm(1-s_r) \Gm(s_r+s_{r+1}-1)}{\Gm(s_{r+1})} \Ll_{MT,r-1} (s_1, \dots,s_{r-1},s_r+s_{r+1}-1; a_1, \dots, a_{r-1}),
\end{align*}
which is a modified function of (\ref{formula:MTML}). Then 
except for the singularity points, it holds that
\begin{align}
\lab{formula:FEforMTLF1}
&\G_r (s_1, \dots, s_r,s_{r+1}  ; a_1, \dots, a_{r-1},1) \\
&=\left(2\pi\right)^{s_r+s_{r+1}-1} \Gm(1-s_r) \left\{ e^{\pi i \frac{s_r+s_{r+1}-1}{2}} F_1^+ (s_1,\dots,s_{r+1}; a_1, \dots, a_{r-1},1) \right. \nonumber \\
&\quad \left. + e^{-\pi i \frac{s_r+s_{r+1}-1}{2}} F_1^- (s_1,\dots,s_{r+1}; a_1, \dots, a_{r-1},1) \right\}. \nonumber 
\end{align}
\end{thm}

Moreover, we obtain a similar formula for the twisted case, where a modified function does not appear. Let 
\begin{equation*}
\kappa_j=\kappa(\chi_j)=\begin{cases}
0 & if \quad \chi_j(-1)=1 \\
1 & if \quad \chi_j(-1)=-1, \\
\end{cases}
\end{equation*}
and $\eps(\chi_r) =\frac{\tau(\chi_r)}{i^{\kappa_r} \sqrt{q}}$, where $\tau(\chi) = \sum_{a=1}^q \chi(a) e^{2 \pi i \frac{a}{q}}$ is the Gauss sum. Then we obtain the following formula.
\begin{thm}
\lab{thm:FEforMTLF2}
We assume that $a_r$ be the primitive Dirichlet character $\chi_r$ modulo $q>1$. Then 
except for the singularity points, it holds that
\begin{align}
\lab{formula:FEforMTLF2}
&\Ll_{MT,r}(s_1,\dots,s_r, s_{r+1}; a_1,\dots,a_{r-1},\chi_r) \\
&= q^{-\frac{1}{2}} \eps(\chi_r) \left( \frac{2\pi}{q} \right)^{s_r+s_{r+1}-1} \Gm(1-s_r) \left\{ e^{\pi i \frac{s_r+s_{r+1}+\kappa_r-1}{2}} F_q^+ (s_1,\dots,s_{r+1} ; a_1, \dots, a_{r-1},\overline{\chi_r})  \right. \nonumber \\
&\quad \left. +e^{-\pi i \frac{s_r+s_{r+1}+\kappa_r-1}{2}} F_q^- (s_1,\dots,s_{r+1} ; a_1, \dots, a_{r-1}, \overline{\chi_r}) \right\}. \nonumber
\end{align}
\end{thm}

Okamoto and Onozuka showed Theorem \ref{thm:FEforMTZF} by using the Hankel contour integral (see {\cite[Section 12.22, p. 245]{WW}}). They generalized Matsumoto's idea \cite{Ma04} to the Mordell-Tornheim multiple zeta-function. 

Our present method is different. In fact, as proved in Lemma \ref{lem:L-MT}, our results can be proved by using the classical Mellin-Barnes integral formula, that is
\begin{equation}
\lab{formula:MBI}
(1+\la)^{-s} = \frac{1}{2 \pi i} \int_{(c)} \frac{\Gm(s+z)\Gm(-z)}{\Gm(s)} \la^z dz,
\end{equation}
where $s, \la \in \mathbb{C}$ with $\si =\re (s) >0, \lvert\arg \la \rvert< \pi, \la \neq 0, c$ is real with $-\si<c<0$, and the path $(c)$ of integration is the vertical line $\re (z) =c$ (see ({\cite[Section 14.51, p. 289, Corollary]{WW}}). This idea comes from the paper of Kiuchi, Tanigawa and Zhai \cite[Section 2]{KiTaZh11}. They pointed out that the Mellin-Barnes integral expression of (\ref{formula:EZ-DZF}) is connected with the following Mellin-Barnes integral expressin of (\ref{confluent hypergeometric function}):
\begin{equation}
\lab{Psi-MBI}
\Psi(a,c;x) = \frac{1}{2 \pi i} \int_{(\gamma)} \frac{\Gm(a+z) \Gm(-z) \Gm(1-c-z)}{\Gm(a) \Gm(a-c+1)} x^z dz
\end{equation}
for $-\re (a)<\gamma<\min \{ 0,1-\re(c) \}, -3\pi/2<\arg x < 3\pi/2$ (\cite[6.5 (5)]{ErMaObTr}).

\begin{rem}
In the case that $a_1(s) = \dots =a_{r-1}(s) =1$ in Theorem \ref{thm:FEforMTLF1}, we can obtain Theorem \ref{thm:FEforMTZF} by combining $\G_r (s_1,\dots,s_r,s_{r+1} ; 1, \dots, 1)$ and $\G_r (-s_1,\dots,-s_{r-1}, 1-s_{r+1},1-s_r ; 1, \dots,1)$. Thus Theorem \ref{thm:FEforMTLF1} gives a generalization of Theorem \ref{thm:FEforMTZF}.
\end{rem}

\begin{rem}
It seems difficult to relax the restriction in our results that $\Ll_r$ must be the Riemann zeta-function or a Dirichlet $L$-function with a primitive character. For example, if we put $\Ll_r$ as a more general $L$-function belonging to the extended Selberg class $\mathcal{S}^\#$, which is a family of fundamental $L$-functions, then $\Ll_r$ satisfies the functional equation of the form $\Lambda (s) = \om \overline{\Lambda(1-\overline{s})}$, where $\Lambda(s) = Q^s \prod_{j=1}^J \Gm(\la_j s+ \mu_j) \Ll_r(s)$ with positive real numbers $Q, \la_j$ and complex numbers $\mu_j, \om$ with $\re(\mu_j)\geq 0$ and $\lvert \om \rvert =1$. By the results of Kaczorowski and Perelli \cite[Theorem 3]{KP99} and $\la$-congecture (see \cite[Conjecture 4.2]{P05}), there would be no other $L$-functions except for the above two functions in $\mathcal{S}^\#$ that $J=1$. However, in order to apply (\ref{Psi-MBI}) (see the proof of Lemma \ref{lem:L-MT}), it must be $J=1$. 
\end{rem}

\section{Some lemmas}\label{sec3}
In this section, we prepare some lemmas. The first one is necessary to ensure that the indicated path of integration $(c)$ can be chosen in the proof of Lemma \ref{lem:analytic continuation of L-MT} and Lemma \ref{lem:L-MT}.
\begin{lem}
The function (\ref{formula:MTML}) is absolutely convergent in the region
\begin{equation}
\lab{abs. cnv. reg}
\sum_{\ell=1}^j \left(\si_{k_\ell} - \al_{k_\ell} \right) + \si_{r+1} >j,
\end{equation}
with $1 \leq k_1<k_2<\dots<k_j \leq r$ for any $j=1,2,\dots,r$. 
\end{lem}
\begin{proof}
This is obtaned by $a_k(n) \ll n^{\al_k+\eps}$ and the same argument as in the proof of \cite[Lemma 2.1]{OO}.
\end{proof}

\begin{lem}
\lab{lem:convolution}
Let $s_j = \si_j +it_j$ for $1\leq  j \leq r+1$ and $b$ be a complex number. Then in the intersection of the regions defined by the inequalities
\begin{equation*}
\sum_{\ell=1}^j \left(\si_{k_\ell}\ - \al_{k_\ell} \right) + \si_{r+1} >j,
\end{equation*}
with $1 \leq k_1<k_2<\dots<k_j \leq r$ for any $j=1,2,\dots,r-1$ and
\begin{equation*}
\si_1+\si_2+\dots+\si_{r+1}>\max \left\{ r+ \sum_{k=1}^r \al_k, \re(b)+1 \right\} ,
\end{equation*}
we have
\begin{align*}
&\Ll_1 (s-b) \Ll_{MT,r}(s_1,\dots, s_r,s_{r+1};a_2,\dots,a_{r+1}) \\
&= \sum_{m_1,\dots,m_r = 1}^\infty \frac{\sum_{n \mid \gcd (m_1, \dots, m_r)} n^b a_2(\frac{m_1}{n}) \dots a_{r+1}(\frac{m_r}{n}) a_1(n)}{m_1^{s_1} \dots m_r^{s_r} (m_1+\dots+m_r)^{s_{r+1}}},
\end{align*}
where $s= s_1+\dots+s_r+s_{r+1}$.
\end{lem}
\begin{proof}
This is a direct genelization of \cite[Lemma 2.2]{OO}.
\end{proof}

\begin{lem}
\lab{lem:analytic continuation of L-MT}
The function (\ref{formula:MTML}) can be continued meromorphically to the whole $\mathbb{C}^{r+1}$ space. Moreover if all $\Ll(s)$ are entire, then $\Ll_{MT,r}$ is entire. 
\end{lem}
\begin{proof}
This can be proved by the same argument as in the proof of Theorem 1 and Theorem 2 in \cite{MaTa03}. So we omit the details. By applying (\ref{formula:MBI}), we can find that 
\begin{align*}
&\Ll_{MT,r}(s_1,\dots,s_r, s_{r+1}; a_1,\dots,a_r) \\
&= \sum_{m_1, \dots,m_r=1}^\infty \frac{a_1(m_1) \dots a_r(m_r)}{m_1^{s_1} \dots m_r^{s_r} \left((m_1 + \dots +m_{r-1}) \left(1+ \frac{m_r}{m_1 + \dots +m_{r-1}}\right) \right)^{s_{r+1}}} \\
&= \frac{1}{2\pi i} \int_{(c)} \frac{\Gm(s_{r+1}+z) \Gm(-z)}{\Gm(s_{r+1})} \Ll_{MT,r-1} (s_1, \dots, s_{r-1},s_{r+1}+z; a_1,\dots,a_{r-1}) \\
&\quad \times \Ll_r(s_r-z) dz, 
\end{align*}
where $\max_{1 \leq j \leq r-1} \{ -\si_r, -\si_r-\si_{r+1}-\sum_{\ell =1}^j (\si_{k_\ell}-\al_{k_\ell}) -j \}<c<-\al_r$ with $1 \leq k_1<k_2<\dots<k_j \leq r-1$. We can choose such $c \in \mathbb{R}$ by Lemma \ref{abs. cnv. reg}. Then we shift the path of integration to the right from $(c)$ to $(M-\eps)$ where $M$ is an arbitrary non-negative integer and $\eps$ is a small positive number, we obtain the meromorphic continuation of (\ref{formula:MTML}) to whole $\mathbb{C}^{r+1}$ space.
\end{proof}

\begin{lem}
\lab{lem:L-MT}
Let $(s_1,\dots, s_r,s_{r+1}) \in \mathbb{C}^{r+1}$ be in the region (\ref{abs. cnv. reg}). Moreover we assume that $\si_r<\al_r, \si_{r+1}>0$. 
\begin{enumerate}[(1)]
\item\lab{lem:L-MT1} If $a_r=1$, then we have
\begin{align}
\lab{formula:L-MT1}
&\Ll_{MT,r}(s_1,\dots,s_r, s_{r+1}; a_1,\dots,a_{r-1},1) \\
&= \frac{\Gm(1-s_r) \Gm(s_r+s_{r+1}-1)}{\Gm(s_{r+1})} \Ll_{MT,r-1} (s_1, \dots,s_{r-1},s_r+s_{r+1}-1; a_1, \dots, a_{r-1}) \nonumber \\
&+ \left( 2\pi \right)^{s_r+s_{r+1}-1} \Gm(1-s_r) \left\{ e^{\pi i \frac{s_r+s_{r+1}-1}{2}} F_1^+ (s_1,\dots,s_{r+1} ; a_1, \dots, a_{r-1},1) \right. \nonumber \\
&\quad \left. + e^{-\pi i \frac{s_r+s_{r+1}-1}{2}} F_1^- (s_1,\dots,s_{r+1} ; a_1, \dots, a_{r-1},1) \right\}. \nonumber 
\end{align}

\item\lab{lem:L-MT2} If $a_r$ is the primitive Dirichlet character $\chi_r$, then we have
\begin{align}
\lab{formula:L-MT2}
&\Ll_{MT,r}(s_1,\dots,s_r, s_{r+1}; a_1,\dots,a_{r-1}, \chi_r) \\
&= q^{-\frac{1}{2}} \eps(\chi_r) \left(\frac{2\pi}{q} \right)^{s_r+s_{r+1}-1} \Gm(1-s_r) \nonumber \\
&\quad \times \left\{ e^{\pi i \frac{s_r+s_{r+1}+\kappa_r-1}{2}} F_1^+ (s_1,\dots,s_{r+1} ; a_1, \dots, a_{r-1},\overline{\chi_r}) \right. \nonumber \\
&\qquad \left.+ e^{-\pi i \frac{s_r+s_{r+1}+\kappa_r-1}{2}} F_q^- (s_1,\dots,s_{r+1} ; a_1, \dots, a_{r-1}, \overline{\chi_r}) \right\}. \nonumber
\end{align}
\end{enumerate}
\end{lem}

\begin{proof}
Let $(s_1,\dots, s_r,s_{r+1}) \in \mathbb{C}^{r+1}$ be in the region (\ref{abs. cnv. reg}). Then we apply (\ref{formula:MBI}) to the expression (\ref{formula:L_MT}) to obtain, 
\begin{align}
&\Ll_{MT,r}(s_1,\dots,s_r, s_{r+1}; a_1,\dots,a_r) \nonumber \\
&= \sum_{m_1, \dots,m_r=1}^\infty \frac{a_1(m_1) \dots a_r(m_r)}{m_1^{s_1} \dots m_r^{s_r} \left( m_r \left(1+ \frac{m_1 + \dots +m_{r-1}}{m_r} \right) \right)^{s_{r+1}}} \nonumber \\
&= \frac{1}{2\pi i} \int_{(c)} \frac{\Gm(s_{r+1}+z) \Gm(-z)}{\Gm(s_{r+1})} \Ll_{MT,r-1} (s_1, \dots, s_{r-1},-z; a_1,\dots,a_{r-1}) \nonumber \\
&\quad \times \Ll_r(s_r+s_{r+1}+z) dz, \nonumber 
\end{align}
where $\max \{ \al_r-\si_r-\si_{r+1}+1, -\si_{r+1} \}<c <\min_{1 \leq j \leq r-1} \{ 0, \sum_{\ell =1}^j (\si_{k_\ell}-\al_{k_\ell}) -j \}$ with $1 \leq k_1<k_2<\dots<k_j \leq r-1$. Assume that $\si_r<\al_r$, then $-\si_{r+1}<\al_r-\si_r-\si_{r+1}<\al_r-\si_r-\si_{r+1}+1$. Take a real number $\eta$ such that $-\si_{r+1}<\eta<\al_r-\si_r-\si_{r+1}$, and move the path of integration to the left from $(c)$ to $(\eta)$. 

(\ref{lem:L-MT1}) If $\Ll_r$ is the Riemann zeta-function, then $\Ll_r(s_r+s_{r+1}+z)$ has a pole at $s=1-s_r-s_{r+1}$ wtih residue $1$. So we obtain
\begin{align*}
&\Ll_{MT,r}(s_1,\dots,s_r, s_{r+1}; a_1,\dots,a_{r-1},1) \\
&= \frac{\Gm(1-s_r) \Gm(s_r+s_{r+1}-1)}{\Gm(s_{r+1})} \Ll_{MT,r-1} (s_1, \dots,s_{r-1},s_r+s_{r+1}-1; a_1, \dots, a_{r-1}) \\
&\quad +\frac{1}{2\pi i} \int_{(\eta)} \frac{\Gm(s_{r+1}+z) \Gm(-z)}{\Gm(s_{r+1})} \Ll_{MT,r-1} (s_1, \dots, s_{r-1},-z; a_1,\dots,a_{r-1}) \\
&\quad \times \Ll_r(s_r+s_{r+1}+z) dz.
\end{align*}
Substituting the functional equation for the Riemann zeta-function (for example \cite[Theorem 12.7]{Apo}):
\begin{equation}
\lab{FE-zeta}
\zeta(1-s) = \zeta(s) 2^{1-s}\pi^{-s}\Gm(s) \cos \frac{\pi s}{2}
\end{equation}
into the above, we have
\begin{align*}
&\Ll_{MT,r}(s_1,\dots,s_r, s_{r+1}; a_1,\dots,a_{r-1},1) \nonumber \\
&= \frac{\Gm(1-s_r) \Gm(s_r+s_{r+1}-1)}{\Gm(s_{r+1})} \Ll_{MT,r-1} (s_1, \dots,s_{r-1},s_r+s_{r+1}-1; a_1, \dots, a_{r-1}) \\
&\quad + \frac{1}{2\pi i} \int_{(\eta)} \frac{\Gm(s_{r+1}+z) \Gm(-z)}{\Gm(s_{r+1})} \Ll_{MT,r-1} (s_1, \dots, s_{r-1},-z; a_1,\dots,a_{r-1}) \nonumber \\
&\quad \times \frac{(2\pi)^{s_r+s_{r+1}+z}}{\pi} \Gm(1-s_r-s_{r+1}-z) \zeta(1-s_r-s_{r+1}-z) \nonumber \\
&\quad \times \cos \frac{\pi (1-s_r-s_{r+1}-z)}{2} dz. \nonumber 
\end{align*}
We apply Lemma \ref{lem:convolution} with $b= s_1+\dots+s_r+s_{r+1}-1$ to find that the above is equal to
\begin{align*}
&\frac{\Gm(1-s_r) \Gm(s_r+s_{r+1}-1)}{\Gm(s_{r+1})} \Ll_{MT,r-1} (s_1, \dots,s_{r-1},s_r+s_{r+1}-1; a_1, \dots, a_{r-1}) \\
&+ \left( 2\pi \right)^{s_r+s_{r+1}-1} \nonumber \sum_{\ell_1,\dots,\ell_{r-1} \geq 1} \frac{\sum_{n \mid \gcd(\ell_1,\dots,\ell_{r-1})} n^{s_1+\dots+s_r+s_{r+1}-1} a_1(\frac{\ell_1}{n}) \dots a_{r-1}(\frac{\ell_{r-1}}{n})}{\ell_1^{s_1} \dots \ell_{r-1}^{s_{r-1}}} \nonumber \\
&\quad \times \frac{1}{2 \pi i} \int_{(\eta)} \frac{\Gm(s_{r+1}+z) \Gm(-z)}{\Gm(s_{r+1})} \Gm(1-s_r-s_{r+1}-z) \nonumber \\
&\quad \times \left( e^{\pi i \left( \frac{s_r+s_{r+1}-1}{2}+\frac{z}{2} \right)}+e^{-\pi i \left( \frac{s_r+s_{r+1}-1}{2}+\frac{z}{2} \right)}\right) \left( 2\pi (\ell_1+\dots +\ell_{r-1})\right)^z dz \nonumber \\
&=\left( 2\pi \right)^{s_r+s_{r+1}-1} \Gm(1-s_r) \\
&\quad \times \sum_{\ell_1,\dots,\ell_{r-1} \geq 1} \frac{\sum_{n \mid \gcd(\ell_1,\dots,\ell_{r-1})} n^{s_1+\dots+s_r+s_{r+1}-1} a_1(\frac{\ell_1}{n}) \dots a_{r-1}(\frac{\ell_{r-1}}{n})}{\ell_1^{s_1} \dots \ell_{r-1}^{s_{r-1}}} \nonumber \\
&\quad \times \left\{ e^{\pi i \frac{s_r+s_{r+1}-1}{2}} \Psi(s_{r+1}, s_r+s_{r+1}; 2\pi i(\ell_1+\dots+\ell_{r-1})) \right. \nonumber \\
&\qquad + \left. e^{-\pi i \frac{s_r+s_{r+1}-1}{2}} \Psi(s_{r+1}, s_r+s_{r+1}; -2\pi i(\ell_1+\dots+\ell_{r-1})) \right\}, \nonumber
\end{align*}
where the last equality comes from (\ref{Psi-MBI}).

(\ref{lem:L-MT2}) If $\Ll_r$ is the Dirichlet $L$-function attached to the primitive character $\chi_r \mod q$, then there is no relevant pole, since $L(s, \chi_r)$ is entire. Instead of substituting (\ref{FE-zeta}), we apply the functional equation for the Dirichlet $L$-function (see \cite[Exercise 12.8]{Apo})
\begin{equation}
\lab{FE-L}
L(1-s, \chi) = \eps(\chi) L(s,\overline{\chi}) 2^{1-s}\pi^{-s}q^{s-\frac{1}{2}} \Gm(s) \cos \frac{\pi (s-\kappa)}{2}
\end{equation}
into the above, we have
\begin{align}
&\Ll_{MT,r}(s_1,\dots,s_r, s_{r+1}; a_1,\dots,a_{r-1}, \chi_r) \nonumber \\
&= \frac{1}{2\pi i} \int_{(\eta)} \frac{\Gm(s_{r+1}+z) \Gm(-z)}{\Gm(s_{r+1})} \Ll_{MT,r-1} (s_1, \dots, s_{r-1},-z; a_1,\dots,a_{r-1}) \nonumber \\
&\quad \times \eps(\chi_r) \frac{(2\pi)^{s_r+s_{r+1}+z}}{\pi} q^{\frac{1}{2}-s_r-s_{r+1}-z} \Gm(1-s_r-s_{r+1}-z) L(1-s_r-s_{r+1}-z,\overline{\chi_r}) \nonumber \\
&\quad \times \cos \frac{\pi (1-s_r-s_{r+1}-z-\kappa_r)}{2} dz. \nonumber 
\end{align}
The remaining argument is the same as in (\ref{lem:L-MT1}).
\end{proof}

\section{Proof of Main Theorems}\label{sec4}
In this section, we complete the proof of Theorem \ref{thm:FEforMTLF1} and Theorem \ref{thm:FEforMTLF2}.
\begin{proof}[Proof of Theorem \ref{thm:FEforMTLF1} and Theorem \ref{thm:FEforMTLF2}]
In order to complete the main results, it suffices to show that $F_x^\pm (s_1,\dots,s_r,s_{r+1};a_1,\dots,a_{r-1},a_r)$ can be continued meromorphically to the whole $\mathbb{C}^{r+1}$ space. The proof is similar to that of \cite[Theorem 3.4]{OO}. By applying the well-known formula \cite[6.5 (6)]{ErMaObTr}
\begin{equation}
\lab{psi-reflection}
\Psi(a,c;x)= x^{1-c} \Psi(a-c+1,2-c;x),
\end{equation}
we obtain
\begin{align}
&F_x^\pm (s_1,\dots,s_r,s_{r+1};a_1,\dots,a_{r-1},a_r) \\
&=\left( \pm \frac{2\pi i}{x} \right)^{1-s_r-s_{r+1}} \nonumber \\
&\quad \times \sum_{\ell_1,\dots,\ell_{r-1} \geq 1} \frac{\sum_{n \mid \gcd(\ell_1,\dots,\ell_{r-1})} n^{s_1+\dots+s_r+s_{r+1}-1} a_1(\frac{\ell_1}{n}) \dots a_{r-1}(\frac{\ell_{r-1}}{n}) a_r(n)}{\ell_1^{s_1} \dots \ell_{r-1}^{s_{r-1}}(\ell_1+\dots+\ell_{r-1})^{s_r+s_{r+1}-1}} \nonumber \\
&\quad \times \Psi(1-s_r, 2-s_r-s_{r+1}; \pm 2\pi i(\ell_1+\dots+\ell_{r-1})/x) . \nonumber
\end{align}
Also, we use the asymptotic expansion (see \cite[6.13.1 (1)]{ErMaObTr})
\begin{equation*}
\Psi(a,c;x) = \sum_{k=0}^{N-1} \frac{(-1)^k (a)_k (a-c+1)_k}{k!}x^{-a-k} + \rho_N (a,c;x),
\end{equation*}
where $N$ is an arbitrary non-negative integer, $(a)_k = \Gm(a+k)/\Gm(a)$ and $\rho_N (a,c;x)$ is the remaider term. Then by Lemma \ref{lem:convolution}, we obtain 
\begin{align}
\lab{formula:expanssion of F_x^pm}
&F_x^\pm (s_1,\dots,s_r,s_{r+1};a_1,\dots,a_{r-1},a_r) \\
&=\sum_{k=0}^{N-1} \frac{(-1)^k (1-s_r)_k (s_{r+1})_k}{ (2\pi/x)^{s_{r+1}+k} k!} \Ll_r(1-s_r+k) \Ll_{MT, r-1} (s_1,\dots,s_{r-1}, s_{r+1}+k; a_1,\dots, a_{r-1}) \nonumber \\
&\quad + \left( \pm \frac{2\pi i}{x} \right)^{1-s_r-s_{r+1}} \sum_{\ell_1,\dots,\ell_{r-1} \geq 1} \frac{\sum_{n \mid \gcd(\ell_1,\dots,\ell_{r-1})} n^{s_1+\dots+s_r+s_{r+1}-1} a_1(\frac{\ell_1}{n}) \dots a_{r-1}(\frac{\ell_{r-1}}{n}) a_r(n)}{\ell_1^{s_1} \dots \ell_{r-1}^{s_{r-1}}(\ell_1+\dots+\ell_{r-1})^{s_r+s_{r+1}-1}} \nonumber \\
&\quad \times \rho_N(1-s_r, 2-s_r-s_{r+1}; \pm 2\pi i(\ell_1+\dots+\ell_{r-1})/x). \nonumber 
\end{align}
The first term on the right hand side of (\ref{formula:expanssion of F_x^pm}) is continued meromorphically to the whole $\mathbb{C}^{r+1}$ space by the assumption (\ref{axiom2}) and Lemma \ref{lem:analytic continuation of L-MT}.

For the second term, by applying the estimate (\cite[(6.2)]{Ma98})
\begin{align*}
&\lvert \rho_N (1-s_r, 2-s_r-s_{r+1}; \pm 2\pi i(\ell_1+\dots+\ell_{r-1})/x) \rvert \\
& \ll \frac{\lvert (s_{r+1})_k \rvert \Gm(-\si_r +N+1)}{N! \lvert \Gm(1-s_r) \rvert} e^{\pi( \lvert t_r \rvert + \lvert t_{r+1} \rvert)/2} (2\pi (\ell_1+\dots+\ell_{r-1})/x)^{\si_r-N-1},
\end{align*}
where $\si_r<N+1$ and $\si_{r+1}\geq -N$, the second term can be estimated as
\begin{align*}
&\ll \frac{\lvert (s_{r+1})_k \rvert \Gm(-\si_r +N+1)}{N! \lvert \Gm(1-s_r) \rvert} e^{\pi( \lvert t_r \rvert + \lvert t_{r+1} \rvert)} \\
&\quad \times \Ll_{MT,r-1}(\si_1,\dots,\si_{r-1},\si_{r+1}+N; \lvert a_1 \rvert, \dots, \lvert a_{r-1} \rvert) \Ll_r (1-\si_r+N, \lvert a_r \rvert).
\end{align*}
By Lemma \ref{abs. cnv. reg} and the assumption (\ref{axiom1}), we see that the second term on the right hand side of (\ref{formula:expanssion of F_x^pm}) is convergent absolutely when $\si_r<N-\al_r$ and 
\begin{equation*}
\sum_{\ell=1}^j \left(\si_{k_\ell} - \al_{k_\ell} \right) + \si_{r+1} >j-N,
\end{equation*}
with $1 \leq k_1<k_2<\dots<k_j \leq r-1$ for any $j=1,2,\dots,r-1$. Since $N$ is arbitrary, $F_x^\pm (s_1,\dots,s_r,s_{r+1};a_1,\dots,a_{r-1},a_r)$ can be continued meromorphically to the whole $\mathbb{C}^{r+1}$ space. Therefore we find that (\ref{formula:FEforMTLF1}) and (\ref{formula:FEforMTLF2}) hold on the whole $\mathbb{C}^{r+1}$ space except for singularity points.
\end{proof}

\section{Special cases}\label{sec5}
In the case of double $L$-functions, our results can be applied to obtain functional equations. As the first application, we invoke the double series considered by Choie and Matsumoto \cite{CM16}. They studied the double series
\begin{equation*}
L_2 (s_1,s_2;\mathfrak{A}) = \sum_{m,n \geq 1} \frac{a(n)}{m^{s_1} (m+n)^{s_2}},
\end{equation*}
where $\mathfrak{A}=\{a(n) \}_{n\geq1}$ is acomplex sequence satisfying (i) $a(n) \ll n^{\frac{\kappa-1}{2}+\eps}$ for a certain constant $\kappa \geq 1$ and arbitrary small positive number $\eps$, (ii) the Dirichlet series $L(s,\mathfrak{A})= \sum_{n \geq 1} a(n) n^{-s}$ can be continued meromorphically to the complex plane and has only finitely many poles. On these assumptions, they showed in \cite[Thorem 2.1]{CM16} that the function
\begin{align*}
F_\pm (s_1,s_2;\mathfrak{A}) &= \sum_{\ell \geq 1} \sum_{n \mid \ell} n^{s_1+s_2-1} a(n) \Psi(s_2,s_1+s_2;\pm 2\pi i  \ell)
\end{align*}
can be continued meromorphically to the $\mathbb{C}^2$ and except for singularity points, it holds 
\begin{align*}
L_2 (s_1,s_2;\mathfrak{A}) &= \frac{\Gm(1-s_1) \Gm(s_1+s_2-1)}{\Gm(s_2)} L (s_1+s_2-1;\mathfrak{A})\\
&\quad+ \Gm(1-s_1) \left\{ F_+ (1-s_2,1-s_1;\mathfrak{A})+F_- (1-s_2,1-s_1;\mathfrak{A}) \right\}.
\end{align*}
We put $\Ll_1 = L(s,\mathfrak{A})$, then $\Ll_1$ satisfies the axioms (\ref{axiom1}) and (\ref{axiom2}). Since $F_1^\pm (0,s_1,s_2; a,1) = F_\pm (s_1,s_2;\mathfrak{A})$ when $r=2$ and $\sum_{n \mid \ell} n^c a(\frac{\ell}{n}) = \ell^c \sum_{n \mid \ell} n^{-c} a(n)$, we see that Theorem \ref{thm:FEforMTLF1} includes \cite[Thorem 2.1]{CM16} since $\Ll_{MT,2}(0,s_1,s_2;a,1)=L_2 (s_1,s_2;\mathfrak{A})$.

Thoerem \ref{thm:FEforMTLF2} also has an application to different double series. Komori, Matsumoto and Tsumura \cite{KMT11} considered the following double $L$-function
\begin{equation*}
L_2 (s_1,s_2;\chi_1,\chi_2) = \sum_{m=1}^\infty \sum_{n=1}^\infty \frac{\chi_1(m) \chi_2(n)}{m^{s_1} (m+n)^{s_2}},
\end{equation*}
where $\chi_1,\chi_2$ are primitive Dirichlet characters modulo $q>1$. In \cite[Corollary 2.3]{KMT11}, they showed that
\begin{equation*}
\left( \frac{2\pi i}{q} \right)^{\frac{1-s_1-s_2}{2}} \frac{\Gm(s_2)}{\tau(\chi_1)} L_2 (s_1,s_2;\chi_1,\chi_2) =\left( \frac{2\pi i}{q} \right)^{\frac{s_1+s_2-1}{2}} \frac{\Gm(1-s_1)}{\tau(\overline{\chi_1})} L_2 (1-s_2,1-s_1;\overline{\chi_2},\overline{\chi_1})
\end{equation*}
holds on the hyperplane $s_1+s_2=2k+1 (k \in \Z)$ if $\chi_1(-1)\chi_2(-1) =1$, and on the hyperplane $s_1+s_2=2k$ ($k \in \mathbb{Z}$) if $\chi_1(-1)\chi_2(-1)=-1$.

Since $\Ll_{MT,2}(0,s_1,s_2;\chi_2,\chi_1)=L_2 (s_1,s_2;\chi_1,\chi_2)$ for primitive characters $\chi_1,\chi_2$ mod $q>1$, we can find that Theorem \ref{thm:FEforMTLF2} is a generalization of \cite[Corollary 2.3]{KMT11}.

\begin{ack} 
The author would like to thank Professor Kohji Matsumoto for valuable comments. Also the auther would like to thank Dr. Sh\={o}ta Inoue for his comments on an earlier version of this article.
\end{ack} 

\bibliographystyle{plain}

\end{document}